\documentclass[10pt]{article}
\font\smallit=cmti10
\font\smalltt=cmtt10

\usepackage{amssymb,latexsym,amsmath,epsfig,amsthm,comment,outlines,enumerate,times,mathrsfs,epsfig,amscd,color,hyperref,url,breakurl,resizegather,tikz,pdflscape,mathtools}
\input amssym.def
\input amssym.tex
\newcommand{\bburl}[1]{\textcolor{blue}{\url{#1}}}
\usepackage[linesnumbered]{algorithm2e}
\usetikzlibrary{arrows,matrix,positioning}
\newcommand{\Rn}{Q_n}
\newcommand{\R}{\ensuremath{\mathbb{R}}}
\newcommand{\N}{\mathbb{N}}
\newcommand{\mP}{\mathcal{P}}
\newcommand{\mB}{\mathcal{B}}
\newcommand{\Cn}{Z}
\renewcommand{\l}{\lambda}

\makeatletter

\renewcommand\section{\@startsection {section}{1}{\z@}
{-30pt \@plus -1ex \@minus -.2ex}
{2.3ex \@plus.2ex}
{\normalfont\normalsize\bfseries\boldmath}}

\renewcommand\subsection{\@startsection{subsection}{2}{\z@}
{-3.25ex\@plus -1ex \@minus -.2ex}
{1.5ex \@plus .2ex}
{\normalfont\normalsize\bfseries\boldmath}}

\renewcommand{\@seccntformat}[1]{\csname the#1\endcsname. }

\makeatother

\theoremstyle{plain}
\newtheorem{theorem}{Theorem}
\newtheorem{lemma}{Lemma}

\newtheorem{proposition}{Proposition}
\newtheorem{corollary}{Corollary}

\theoremstyle{definition}
\newtheorem{definition}{Definition}
\newtheorem{example}{Example}

\begin{document}

\begin{center}
\uppercase{\bf The bidirectional ballot polytope}
\vskip 20pt
{\bf Steven J. Miller}\\
{\smallit Department of Mathematics and Statistics, Williams College, Williamstown, Massachusetts}\\
{\tt sjm1@williams.edu; Steven.Miller.MC.96@aya.yale.edu}\\
\vskip 10pt
{\bf Carsten Peterson}\\
{\smallit Department of Mathematics, University of Michigan, Ann Arbor, Michigan}\\
{\tt carstenp@umich.edu}\\
\vskip 10pt
{\bf Carsten Sprunger}\\
{\smallit Department of Mathematics, Stanford University, Palo Alto, California}\\
{\tt csprun@stanford.edu}\\
\vskip 10pt
{\bf Roger Van Peski}\\
{\smallit Department of Mathematics, Massachusetts Institute of Technology, Cambridge, Massachusetts}\\
{\tt rvp@mit.edu}\\
\end{center}
\vskip 30pt
\centerline{\smallit Received: , Revised: , Accepted: , Published: } 
\vskip 30pt

\centerline{\bf Abstract}
\noindent
A bidirectional ballot sequence (BBS) is a finite binary sequence with the property that every prefix and suffix contains strictly more ones than zeros. BBS's were introduced by Zhao, and independently by Bosquet-M{\'e}lou and Ponty as $(1,1)$-culminating paths. Both sets of authors noted the difficulty in counting these objects, and to date research on bidirectional ballot sequences has been concerned with asymptotics. We introduce a continuous analogue of bidirectional ballot sequences which we call bidirectional gerrymanders, and show that the set of bidirectional gerrymanders form a convex polytope sitting inside the unit cube, which we refer to as the bidirectional ballot polytope. We prove that every $(2n-1)$-dimensional unit cube can be partitioned into $2n-1$ isometric copies of the $(2n-1)$-dimensional bidirectional ballot polytope. Furthermore, we show that the vertices of this polytope are all also vertices of the cube, and that the vertices are in bijection with BBS's. An immediate corollary is a geometric explanation of the result of Zhao and of Bosquet-M{\'e}lou and Ponty that the number of BBS's of length $n$ is $\Theta(2^n/n)$.

\pagestyle{myheadings}
\markright{\smalltt INTEGERS: 18 (2018)\hfill}
\thispagestyle{empty}
\baselineskip=12.875pt
\vskip 30pt

\section{Introduction}\label{sec:intr}

In \cite{Zh1}, Zhao introduced a family of combinatorial objects called bidirectional ballot sequences, defined as follows.

\begin{definition}
A finite 0-1 sequence is a \textbf{bidirectional ballot sequence} (\textbf{BBS}) if every prefix and every suffix contains strictly more ones than zeros. Let $B_n$ denote the number of bidirectional ballot sequences of length $n$.
\end{definition}

Bidirectional ballot sequences have a natural interpretation in terms of lattice paths. Suppose we start at $(0, 0)$ and take a finite number of steps either of the form $(1, 1)$ or $(1, -1)$. We call such a path a \textbf{standard lattice path}. We define the length of the path to be the number of steps it contains. We define the height of a point in the lattice path to be its $y$-coordinate. Bidirectional ballot sequences of length $n$ are in bijection with standard lattice paths of length $n$ whose unique minimum height is attained at the first point in the path, and whose unique maximum height is attained at the last point in the path. The bijection is given by identifying the digit `0' in a BBS with a step of the form $(1, -1)$ and the digit `1' with a step of the form $(1, 1)$ (for an example of this, see Section \ref{section:polytope_vertices}).

From this perspective, bidirectional ballot sequences were independently introduced by Bosquet-M{\'e}lou and Ponty \cite{BP} as a special type of what they call culminating paths. In particular, an $(a, b)$-culminating path is a sequence of lattice points starting at $(0, 0)$ such that each step is of the form $(1, a)$ or $(1, -b)$ and such that the unique minimum height is achieved at the first point and the unique maximum height is achieved at the last point. Thus bidirectional ballot sequences are in bijection with $(1, 1)$-culminating paths. In \cite{BP} it is noted that $(1, 1)$-culminating paths had been used in \cite{FGK} with connections to theoretical physics, and general $(a, b)$-culminating paths had been used in \cite{AGMML}, \cite{CR}, and \cite{PL} with connections to bioinformatics.

In both \cite{Zh1} and \cite{BP}, it is noted that unlike other easy to define classes of lattice paths (e.g. Dyck paths), the enumeration of BBS's is tricky; there is no obvious recursive structure to such paths. Both  authors focused on the asymptotics of $B_n$. In particular, \cite{BP} obtained a generating function in $n$ for the number of $(a, b)$-culminating paths of length $n$ with fixed height $k$ (the generating function for the $(1, 1)$ case was found in \cite{FGK}). Furthermore, they showed that $B_n \sim 2^n/4n$. Independently, \cite{Zh1} showed that $B_n = \Theta (2^n/n)$ and stated without detailed proof that $B_n \sim 2^n/4n$. Additionally in \cite{Zh1}, the author conjectured an even finer asymptotic expression for $B_n$. This conjecture was later proved by  Hackl, Heuberger, Prodinger and Wagner \cite{HHPW}, who refined the asymptotic expression even further using techniques from analytic combinatorics.

The motivation for the study of culminating paths in \cite{BP} was the observation that such paths had been independently introduced and utilized in disparate contexts (theoretical physics and bioinformatics) as well as a general interest in understanding subfamilies of lattice paths. However, the motivation in \cite{Zh1}, as well as our original motivation for studying BBS's, arises from additive combinatorics. Let $A \subset \mathbb{Z}$ be a finite set of integers. We define the sumset $A + A$ as those elements in $\mathbb{Z}$ expressible as $a + b$ with $a, b \in A$. Similarly, the difference set $A - A$ is those elements expressible as $a - b$ with $a, b \in A$. We say that $A$ is a \textbf{more sums than differences} (\textbf{MSTD}) set if $|A + A| > |A - A|$. Because of the commutativity of addition, one may intuitively expect that in general $|A - A| \geq |A + A|$. This intuition turns out to be correct in some contexts (see \cite{HM}), in particular if each element in $[n] := \{1,2,\ldots,n\}$ is independently chosen to be in $A$ with some probability $p(n)$ tending to zero). Let $\rho_n$ be the proportion of subsets of $[n]$ which are MSTD. In \cite{MO}, it was shown that $\rho_n > 2 \times 10^{-7}$ for $n \geq 15$, and in \cite{Zh2} it was shown that $\lim_{n \to \infty} \rho_n$ converges to a positive number; experimental data suggests this limit to be of order $10^{-4}$. Thus, in this sense, a positive proportion of sets are MSTD. However, the techniques in \cite{MO} are probabilistic, and to date no known constant density family of MSTD subsets of $[n]$ as $n \to \infty$ is known.

The best density explicit construction of MSTD sets is due to Zhao in \cite{Zh1} using BBS's. Let $B$ be a binary sequence of length $n$. We can associate to $B$ the set $A \subseteq [n]$ defined as $A := \{i : B_i = 1\}$. For example if $B = 01101$, then $A = \{2, 3, 5\}$. Those subsets $A$ of $[n]$ arising from BBS's have the property that $A + A = \{i : 2 \leq i \leq 2n\}$, which is to say that the sumset is as large as possible (similarly it turns out that the difference set is also as large as possible). Using this property, Zhao was able to translate those subsets of $[n]$ arising from BBS's and append extra elements to the fringes to obtain an MSTD set for each set arising from a BBS. From this, one immediately gets a density $\Theta(1/n)$ family of MSTD sets.

Motivated by the use of BBS's in additive combinatorics, in this paper we study the natural analgoue of BBS's in a continuous setting, which we call \textbf{bidirectional gerrymanders}; in the related paper \cite{MP}, we use similar ideas as in this paper to study the analogue of MSTD sets in a continuous setting.

We first set some notation and then describe our main results. Let $\mathbb{I}_n$ denote the set of all subsets of $\mathbb{R}$ consisting of exactly $n$ disjoint open intervals such that the leftmost interval starts at 0. Suppose $\mathcal{A} \in \mathbb{I}_n$. If we translate $\mathcal{A}$, then the sumset and difference set merely translate as well. Thus, when studying additive behavior, we do not lose any generality by restricting our attention to collections of intervals such that the leftmost interval starts at zero. We can topologize $\mathbb{I}_n$ by identifying it with $\mathbb{R}^{2n-1}_{\geq 0}$, the non-negative orthant: let $\mathcal{A}= I_1 \cup I_2 \cup \dots \cup I_n \in \mathbb{I}_n$ with $I_i$ to the left of $I_j$ for $i < j$. Suppose $I_j = (a_j, b_j)$. We then identify $\mathcal{A}$ with the vector $v_{\mathcal{A}} = [b_1 - a_1, a_2 - b_1, b_2 - a_2, a_3 - b_2, \dots, b_n - a_n]$. Thus the first entry is the length of the first interval, the second entry is the size of the gap between the first and second intervals, the third entry is the length of the second interval, etc. We shall find it convenient to restrict our attention to the following set: let $\mathbb{J}_n \subset \mathbb{I}_n$ be the set of collections of $n$ non-overlapping intervals such that the leftmost interval starts at zero, the length of each interval is between 0 and 1, and the gap between adjacent intervals is between 0 and 1 (if we scale $\mathcal{A} \in \mathbb{I}_n$ by $\alpha \neq 0$, then the sumset and difference set scale by $\alpha$ as well, so $\alpha \mathcal{A}$ has the same essential additive behavior as $\mathcal{A}$; note that up to scaling, every element of $\mathbb{I}_n$ is an element of $\mathbb{J}_n$). We can topologize $\mathbb{J}_n$ by identifying it with $C_{2n-1} = [0, 1]^{2n-1}$, the $2n-1$ dimensional unit cube\footnote{Because the endpoints of an open interval cannot be equal, strictly speaking we are taking $\mathbb{I}_n$ to be the set of all weakly increasing $2n$-tuples of points on the real line and identifying these with collections of $n$ intervals by treating them as endpoints (and correspondingly for $\mathbb{J}_n$). However, in the edge case when $a_j=b_j$, we still allow an `empty' interval at $a_j$, which is included in the data of an element of $\mathbb{I}_n$. Including these degenerate cases allows us to indeed identify $\mathbb{J}_n$ with the closed unit cube.}. For other ways to topologize $\mathbb{I}_n$ and related spaces, see \cite{MP}.

The bidirectional gerrymanders in $\mathbb{J}_n$ form a convex, compact polytope contained in $C_{2n-1}$ which we call the \textbf{bidirectional ballot polytope}, $\mathcal{P}_n$. This polytope has a number of extraordinary combinatorial features. In Section \ref{sec:vol_polytope} we formally define this polytope and show that $C_{2n-1}$ can be partitioned into $2n-1$ disjoint isometric copies of $\mathcal{P}_n$, which in particular shows that the volume of $\mathcal{P}_n$ is $1/(2n-1)$. In Section \ref{section:cube_vertices} we show that the vertices of $\mathcal{P}_n$ are vertices of $C_{2n-1}$. Finally in Section \ref{section:polytope_vertices} we show that the vertices of $\mathcal{P}_n$ are in bijection with $B_{2n+3}$, and that a particular subset of the vertices are in bijection with $B_{2n-1}$. From this we are able to immediately rederive geometrically that $|B_n| = \Theta (2^n/n)$, i.e., there are positive constants $\alpha$ and $\beta$ such that for all $n$ sufficiently large we have $\alpha 2^n/n \le |B_n| \le \beta 2^n/n$.

\section{The Bidirectional Ballot Cone and Polytope} \label{sec:vol_polytope}

We first set some notation. Let $m = 2n-1$ for some $n \in \mathbb{N}$.
\begin{definition}\label{def:bal_vecs}
Let the set of \textbf{left ballot vectors}, $L_n$, and the set of \textbf{right ballot vectors}, $R_n$, be the following sets of vectors in $\mathbb{R}^m$:
\begin{gather}
    L_n := \{[1, -1, 0, \dots, 0], [1, -1, 1, -1, 0, \dots, 0], \dots, [1, -1, \dots, 1, -1, 0]\}, \\
    R_n := \{[0, \dots, 0, -1, 1], [0, \dots, 0, -1, 1, -1, 1], \dots, [0, -1, 1, \dots, -1, 1] \}.
\end{gather}
We define $V_n$, the set of \textbf{ballot vectors}, as $V_n = L_n \cup R_n$.
\end{definition}

\begin{definition}
The \textbf{bidirectional ballot cone}, $\mathcal{B}_n$, is the set of $x \in \mathbb{R}^m$ such that $x \cdot w \geq 0$ for all $w \in V_n$. When the value of $n$ is obvious, we simply refer to it as $\mathcal{B}$.
\end{definition}

We now define the continuous analogue of BBS's, and show in Proposition \ref{prop:zhao_analogue} that it is the right generalization.

\begin{definition}
Let $\mathcal{A} \in \mathbb{I}_n$. We call $\mathcal{A}$ a \textbf{bidirectional gerrymander} if $v_{\mathcal{A}} \in \mathcal{B}$.
\end{definition}

\begin{proposition}\label{prop:zhao_analogue}
Suppose $\mathcal{A} = I_1 \cup \dots \cup I_n \in \mathbb{I}_n$ with endpoints ordered as before. Suppose the right endpoint of $I_n$ is $b$. Then, $\mathcal{A}$ is a bidirectional gerrymander if and only if $\mu(\mathcal{A} \cap [0, t]) \geq t/2$ and $\mu(\mathcal{A} \cap [b-t, b]) \geq t/2$ for all $t \in [0, b]$.
\end{proposition}

\begin{proof}
Clearly if these measure conditions hold, then $\mathcal A$ is a bidirectional gerrymander, as setting $t$ to be left and right endpoints of the $I_i$ yields the nonnegativity conditions of pairing with the ballot vectors. The condition $\mu(\mathcal{A} \cap [0, t]) \geq t/2$ is equivalent to the non-negativity of $\mu(\mathcal{A} \cap [0,t]) - \mu((\R \setminus \mathcal{A}) \cap [0,t])$. For $t \in [0,b]$, $\mu(\mathcal{A} \cap [0,t]) - \mu((\R \setminus \mathcal{A}) \cap [0,t])$ takes a local minimum only if $t$ is a left endpoint of an interval $I_i$. Hence if $v_\mathcal{A} \cdot w \geq 0$ for all $w \in L_n$, then the function is nonnegative at its minima and so the first measure condition holds. Similarly, the second measure condition holds as well by the nonnegativity of pairing with the right ballot vectors.
\end{proof}

A BBS in the sense of \cite{Zh1} is a binary sequence for which any subsequence truncated on the left or right contains more $1$'s than $0$'s, and Proposition \ref{prop:zhao_analogue} shows that a bidirectional gerrymander is a subset of $\mathbb{R}$ contained in $[0, a]$ for which any subset obtained by truncating on the left or right contains ``more'' points (in a measure theoretic sense) in the original set than points not in this set. It is thus clear that they are a natural analogue, but, as we shall see, what is surprising is that they can be used to prove results about standard (discrete) BBS's.

\begin{definition}
The \textbf{bidirectional ballot polytope} $\mathcal{P}_n$, is defined as $\mathcal{B}_n \cap C_{m}$. Equivalently, it is those vectors $v_{\mathcal{A}}$ such that $\mathcal{A} \in \mathbb{J}_n$ is a bidirectional gerrymander. When the value of $n$ is obvious, we shall refer to it simply as $\mathcal{P}$.
\end{definition}

\begin{figure}[h]
\begin{center}
\includegraphics[scale=0.5]{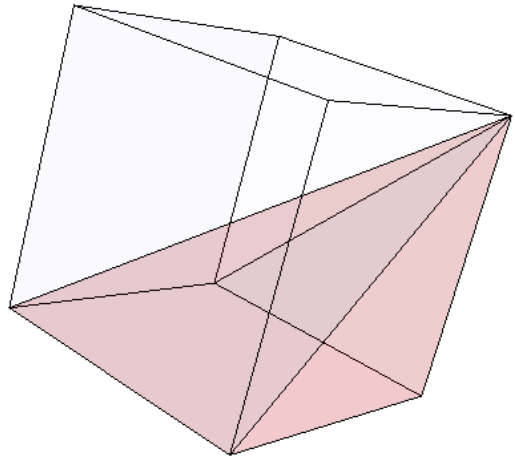}
\caption{The polytope $\mP_2$ (red) sitting inside $C_3$. Notice that adding two additional copies of $\mP_2$, rotated about the main diagonal of the cube by $2 \pi/3$ and $4 \pi/3$ respectively, would result in a partition of $C_3$ (neglecting overlap of boundaries).}
\end{center}
\end{figure}

\begin{definition}
Let $Z_m$ be the cyclic group of order $m$ with generator $\rho$. Let $Z_m$ act on $\mathbb{R}^{m}$ by cyclically permuting the entries (e.g. $\rho^2([0, 1, 2, 3, 4]) = [3, 4, 0, 1, 2]$). For a given set of vectors $V$ and $\sigma \in Z_m$, let $\sigma(V) := \{\sigma(v): v \in V\}$ with $\sigma \in Z_m$. For each $\sigma \in Z_m$, define $\mathcal{B}_\sigma$ by
\begin{gather}
\mathcal{B}_\sigma\ :=\ \{v \in \mathbb{R}^m_{\geq 0} : v \cdot w \geq 0 \ \textnormal{for all} \ w \in \sigma(V_n)\},
\end{gather}
and $\mathcal{P}_\sigma$ likewise. Note that $\mathcal{B}_\sigma = \sigma^{-1}(\mathcal{B})$, and that $\mathcal{B} = \mathcal{B}_{\textnormal{Id}}$ and $\mP = \mP_{\textnormal{Id}}$.
\end{definition}

\begin{theorem}\label{thm:permute_polytope}
The non-negative orthant, $\mathbb{R}^{m}_{\geq 0}$, is contained in $\bigcup_{\sigma \in Z_m} \mathcal{B}_\sigma$. Furthermore, for $\sigma_1 \neq \sigma_2$, the interiors of $\mathcal{B}_{\sigma_1}$ and $\mathcal{B}_{\sigma_2}$ are disjoint.
\end{theorem}

\begin{proof}
Let $\tau = \rho^2 \in \Cn_{m}$ be the cyclic shift by two places. Because $m$ is odd, $\tau$ generates $\Cn_{m}$. In particular, we see that the set of left and right ballot vectors $V_n$ as defined in Definition \ref{def:bal_vecs} is equal to
\begin{equation}
V_n\ =\ \left\{\sum_{i=0}^k \tau^{i}(w): 0 \leq k \leq 2n-3\right\},
\end{equation}
where $w=[1,-1,0,\ldots,0]$.
If $\ell < k \leq 2n-3$ then
\begin{equation}
\sum_{i=0}^k \tau^{i}(w)-\sum_{i=0}^\ell \tau^{i}(w)\ =\ \sum_{i=\ell+1}^k \tau^{i}(w)\ =\ \tau^{\ell+1} \sum_{i=0}^{k-\ell-1} \tau^{i}(w),
\end{equation}
and since $\sum_{i=0}^{2n-2} \tau^{i}(w) = [0,\ldots,0]$ we have similarly that, for $0 \leq k \leq \ell$,
\begin{equation}
\sum_{i=0}^k \tau^{i}(w)-\sum_{i=0}^\ell \tau^{i}(w)\ =\ \tau^{\ell+1} \sum_{i=0}^{(2n-2)+(k-\ell)} \tau^{i}(w).
\end{equation}

Then for each $\ell$ we have that
\begin{equation}
\left\{\sum_{i=0}^k \tau^{i}(w)-\sum_{i=0}^\ell \tau^{i}(w): 0 \leq k \leq 2n-2, k \neq \ell\right\} \ = \ \tau^{\ell+1}(V_n).
\end{equation}
Now let $w_k = \sum_{i=0}^k \tau^{i}(w)$, take any $v \in [0,1]^{m}$, and choose $0 \leq \ell \leq 2n-2$ minimizing $v \cdot w_\ell$ (this $\ell$ may not be unique). Then
\begin{equation}
v \cdot \left(\sum_{i=0}^k \tau^{i}(w)-\sum_{i=0}^\ell \tau^{i}(w)\right)\ \geq\ 0
\end{equation}
for all $0 \leq k \leq 2n-2$. Therefore $v \cdot r \geq 0$ for all $r \in \tau^{\ell+1}(V_n)$, so $v \in \mB_{\tau^{\ell+1}}$. This shows that $\R_{\geq 0}^{m} = \bigcup_{\sigma \in \Cn_{m}} \mB_\sigma$. Intersecting with $C$ gives the corresponding result for $\mP$.

Conversely, if $v \in \text{Int}(\mB_{\tau^{\ell+1}}) \cap \text{Int}(\mB_{\tau^{k+1}})$ and $\tau^{\ell+1} \neq \tau^{k+1}$, then (because taking the interior simply changes the inequalities defining $\mB_{\tau^{\ell+1}}$ to strict ones) we have both
\begin{align*}
    v \cdot \left(\sum_{i=0}^k \tau^{i}(w)-\sum_{i=0}^\ell \tau^{i}(w)\right)\ >\ 0 \\
    v \cdot \left(\sum_{i=0}^\ell \tau^{i}(w)-\sum_{i=0}^k \tau^{i}(w)\right)\ >\ 0.
\end{align*}
This is a contradiction, so the interiors distinct regions $\mB_{\tau^{\ell+1}}$ are disjoint, and it follows immediately that the interiors of distinct regions $\mP_{\tau^{\ell+1}}$ are disjoint.
\end{proof}

\begin{corollary}\label{cor:volume}
The unit cube $C_m$ equals $\bigcup_{\sigma \in Z_m} \mathcal{P}_\sigma$. Furthermore, for $\sigma_1 \neq \sigma_2$, the interiors of $\mathcal{P}_{\sigma_1}$ and $\mathcal{P}_{\sigma_2}$ are disjoint. Consequently, the volume of $\mathcal{P}$ is exactly $1/m$.
\end{corollary}

\begin{proof}
Intersecting the nonnegative orthant and the translates $\mathcal{B}_\sigma$ with $C_m$, Theorem \ref{thm:permute_polytope} yields that $C_m$ is partitioned into $m$ regions produced by permuting the coordinates of $\mP$. Because the matrix representing $\tau = \rho^2$ has determinant $1$ it leaves volume invariant. Therefore, $\text{Vol}(\mathcal{P}_\sigma) =\text{Vol}(\mathcal{P})$ for all $\sigma \in Z_m$, so  $\text{Vol}(\mathcal{P}) = 1/m$.
\end{proof}

\begin{corollary}\label{cor:necklace}
For any vector $v \in \R_{\geq 0}^{m}$, there exists $\sigma \in Z_{m}$ such that the vector

\noindent $v'=(v_1',v_2',\ldots,v_{m}')=\sigma(v)$ has the following property: For all $1 \leq k \leq n$,
\begin{equation}\label{eq:necklace1}
\sum_{i=1}^k (v_{2i-1}' - v'_{2i})\ \geq\ 0
\end{equation}
and
\begin{equation}\label{eq:necklace2}
\sum_{i=1}^k (v'_{2n-(2i-1)} - v'_{2n-2i}) \ \geq\ 0.
\end{equation}
If furthermore these are all positive, then $\sigma$ is unique.
\end{corollary}

One interpretation of the above corollary is as follows. Suppose you have a necklace with an odd number of beads. On each bead you write a non-negative number. Then there exists some place where you can cut the necklace such that when you lay out the necklace and think of the sequence of values on the beads as a vector in $\mathbb{R}^{m}$, this vector is a bidirectional gerrymander. Furthermore, if the numbers you write on the beads are ``generic'', in the sense that the inequalities corresponding to \eqref{eq:necklace1} and \eqref{eq:necklace2} are strict, then there is exactly one such place you can cut the necklace.

\begin{center}
\begin{figure}[h]
\begin{tikzpicture}[scale = 1.2]
\draw (0, 1) circle [radius=0.4];
\draw ({-0.951*(1/3)}, {2/3*1 + (0.309)*(1/3)}) -- (-2/3*0.951, {1/3 + (0.309)*2/3});
\node at (0, 1) {\small 1.78};
\draw (-0.951, 0.309) circle [radius=0.4];
\draw ({2/3*(-0.951) + 1/3*(-0.588)}, {2/3*(0.309)+1/3*(-0.809)}) -- ({1/3*(-0.951) + 2/3*(-0.588)}, {1/3*(0.309) + 2/3*(-0.809)});
\node at (-0.951, 0.309) {\small 1.55};
\draw (-0.588, -0.809) circle [radius=0.4];
\draw ({2/3*(-0.588) + 1/3*(0.588)}, {2/3*(-0.809)+1/3*(-0.809)}) -- ({1/3*(-0.588) + 2/3*(0.588)}, {1/3*(-0.809) + 2/3*(-0.809)});
\node at (-0.588, -0.809) {\small 0.76};
\draw (0.588, -0.809) circle [radius=0.4];
\node at (0.588, -0.809) {\small 2.06};
\draw ({2/3*(0.588) + 1/3*(0.951)}, {2/3*(-0.809)+1/3*(0.309)}) -- ({1/3*(0.588) + 2/3*(0.951)}, {1/3*(-0.809) + 2/3*(0.309)});
\draw (0.951, 0.309) circle [radius=0.4];
\node at (0.951, 0.309) {\small 3.21};
\draw ({2/3*(0.951) + 1/3*(0)}, {2/3*(0.309)+1/3*(1)}) -- ({1/3*(0.951) + 2/3*(0)}, {1/3*(0.309) + 2/3*(1)});

\draw [dashed] ({0.5*(0.951 + 0.588) - 0.3*1}, {0.5*(0.309 - 0.809) + 0.3*((0.951 - 0.588)/(0.309 + 0.809))}) -- ({0.5*(0.951 + 0.588) + 0.3*1}, {0.5*(0.309 - 0.809) - 0.3*((0.951 - 0.588)/(0.309 + 0.809))});

\draw [->] (2, 0) -- (3, 0);

\draw (4, 0) circle [radius = 0.4];
\draw (5.2, 0) circle [radius = 0.4];
\draw (6.4, 0) circle [radius = 0.4];
\draw (7.6, 0) circle [radius = 0.4];
\draw (8.8, 0) circle [radius = 0.4];

\node at (4, 0) {\small 3.21};
\node at (5.2, 0) {\small 1.78};
\node at (6.4, 0) {\small 1.55};
\node at (7.6, 0) {\small 0.76};
\node at (8.8, 0) {\small 2.06};

\draw (3.4, 0) -- (3.6, 0);
\draw (4.4, 0) -- (4.8, 0);
\draw (5.6, 0) -- (6.0, 0);
\draw (6.8, 0) -- (7.2, 0);
\draw (8.0, 0) -- (8.4, 0);
\draw (9.2, 0) -- (9.4, 0);
\end{tikzpicture}
\caption{An example ``cut'' of a necklace as in Corollary \ref{cor:necklace}.}
\end{figure}
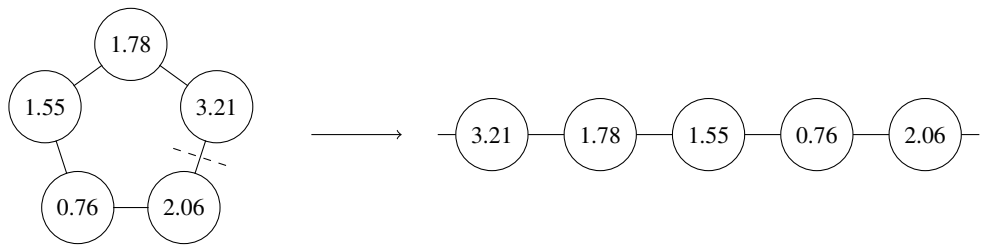
\end{center}

\section{Vertices of the Bidirectional Ballot Polytope are Vertices of the Cube} \label{section:cube_vertices}

In this section we show that the vertices of $\mathcal{P}_n$ are also vertices of $C_m$, the unit cube. We had previously defined $\mathcal{P}_n$ as the intersection of the unit cube with the ballot cone, which is equivalent to the set of vectors $[\ell_1, g_1, \dots, g_{n-1}, \ell_n]$ satisfying the below inequality:

\begin{gather}
\label{big_matrix}
\begin{array}{@{}r@{}l}
\begin{array}[]{@{}r@{}r}
	\text{cube vectors} & \left. \begin{array}{c} \vphantom{1} \\ \vphantom{-1} \\ \vphantom{1} \\ \vphantom{-1} \\ \vphantom{\vdots} \end{array} \right\{ \\
	\text{left ballot vectors} & \left. \begin{array}{c} \vphantom{1} \\ \vphantom{1} \\ \vphantom{\vdots} \end{array} \right\{ \\
	\text{right ballot vectors} & \left. \begin{array}{c} \vphantom{1} \\ \vphantom{1} \\ \vphantom{\vdots} \end{array} \right\{ \\
\end{array}
&
\left[
\begin{array}{c c c c c c c c c c c}
1 & 0 & 0 & 0 & 0 & \dots & 0 & 0 & 0 & 0 & 0 \\
-1 & 0 & 0 & 0 & 0 & \dots & 0 & 0 & 0 & 0 & 0 \\
0 & 1 & 0 & 0 & 0 & \dots & 0 & 0 & 0 & 0 & 0 \\
0 & -1 & 0 & 0 & 0 & \dots & 0 & 0 & 0 & 0 & 0 \\
\vdots & \vdots & \vdots & \vdots & \vdots & \ddots & \vdots & \vdots & \vdots & \vdots & \vdots \\
1 & -1 & 0 & 0 & 0 & \dots & 0 & 0 & 0 & 0 & 0 \\
1 & -1 & 1 & -1 & 0 & \dots & 0 & 0 & 0 & 0 & 0 \\
\vdots & \vdots & \vdots & \vdots & \vdots & \ddots & \vdots & \vdots & \vdots & \vdots & \vdots \\
0 & 0 & 0 & 0 & 0 & \dots & 0 & 0 & 0 & -1 & 1 \\
0 & 0 & 0 & 0 & 0 & \dots & 0 & -1 & 1 & -1 & 1 \\
\vdots & \vdots & \vdots & \vdots & \vdots & \ddots & \vdots & \vdots & \vdots & \vdots & \vdots \\
\end{array}
\right]
\end{array}
\begin{bmatrix}
\ell_1 \\
g_1 \\
\ell_2 \\
g_2 \\
\vdots \\
g_{n-1} \\
\ell_n
\end{bmatrix}
\ \geq \
\begin{bmatrix}
0 \\
-1 \\
0 \\
-1 \\
\vdots \\
0 \\
0 \\
\vdots \\
0 \\
0 \\
\vdots
\end{bmatrix}.
\end{gather}

The first collection of rows in the above matrix is necessary to ensure that we only deal with points inside of the unit cube. Thus we call any vector of the form $[0, \dots, 0, \pm 1, 0, \dots, 0]$ a \textbf{cube vector}.

Before proving the main result of this section, we must review a few concepts related to convex polytopes. We follow the terminology of \cite{BT}.

\begin{definition}
Let $P$ be a polytope in $\mathbb{R}^n$ defined by the inequalities $a_i^T x \geq b_i$ for $i \in \{1, 2, \dots, k\}$. Let $x^*$ be such that for some $i$, $a_i^T x^* = b_i$. Then, we say that the $i$\textsuperscript{th} constraint is \textbf{active} at $x^*$.
\end{definition}

\begin{definition}
A vector $x^* \in \mathbb{R}^n$ is called a \textbf{basic solution} if out of all of the constraints that are active at $x^*$, there is some collection of $n$ of them which is linearly independent. If $x^*$ is a basic solution that satisfies all of the constraints, then it is called a \textbf{basic feasible solution}.
\end{definition}

Part of what makes the study of convex polytopes interesting is that there are several equivalent but strikingly different ways of defining what the vertices of a polytope are. In particular, one definition is that a point $v$ is a vertex if and only if it is a basic feasible solution.

The following shorthand will be helpful in the proof of the main theorem of this section.

\begin{definition}
A matrix/vector is called \textbf{flat} if all of its entries are 0, 1, or -1.
\end{definition}

Let $\Rn$ denote the set of vertices in the polytope $\mathcal{P}_n$. Let $S_n$ denote the set of vertices of the unit cube $C_{m}$. The main result of this section is the following.

\begin{theorem}\label{theorem:cube}
All of the vertices of the bidirectional ballot polytope $\mathcal{P}_n$ are also vertices of the unit cube $C_{m}$; i.e., $\Rn \subset S_n$.
\end{theorem}

\begin{proof}
By the above discussion, we know that we must show that all basic feasible solutions are vertices of the cube. Throughout this proof, we let $n$ be fixed, and let $m = 2n-1$. Thus we unambiguously let $\mathcal{P} = \mathcal{P}_n$, $C = C_{2n-1}$, $Q = \Rn$, and $S = S_n$. Notice that $\mathbb{Z}^m \cap \mathcal{P} \subset S$. From this observation, we now describe the strategy for proving the theorem. Suppose $x^*$ is a basic solution whose corresponding constraints are $a_{i_1}$, $\dots$, $a_{i_m}$. Then $x^*$ satisfies
\begin{gather}
\label{matrix:A}
\begin{bmatrix}
\textnormal{---} a_{i_1} \textnormal{---} \\
\vdots  \\
\textnormal{---} a_{i_m} \textnormal{---}
\end{bmatrix}
x^*\ =\ \begin{bmatrix}
b_{i_1} \\
\vdots \\
b_{i_m}
\end{bmatrix}.
\end{gather}
Let $A$ be the matrix in \eqref{matrix:A}. Let $b^*$ be the vector on the right hand side in \eqref{matrix:A}. Thus $x^* = A^{-1} b^*$. Note that $b^* \in \mathbb{Z}^m$ since it is some subset of the entries in the vector on the right hand side of \eqref{big_matrix}. If we can show that $\det(A) = \pm 1$, it will imply that $A^{-1}$ has integer entries, and thus that $A^{-1} b^* \in \mathbb{Z}^m$. From the earlier observation, if $x^*$ is a basic feasible solution, then we must have that $A^{-1} b^* = x^* \in S$, which would prove the theorem.

Now we must show that if $A$ is invertible, then it has determinant $\pm 1$. In order to show this, we keep track of what happens to the determinant in the process of carrying out Gaussian elimination, which converts $A$ into the identity matrix. In particular, we show that at every step, the determinant changes by a factor of $\pm 1$. Since the identity matrix has determinant 1, we could then conclude that $A$ has determinant $\pm 1$. The only elementary row operation which potentially changes the absolute value of the determinant of a matrix is multiplying a row by a scalar. Thus it suffices to show that when Gaussian elimination is performed on $A$, no row is ever multiplied by a scalar other than $\pm 1$. In Gaussian elimination, a row is multiplied by a scalar to convert some non-zero entry in that row to a one. If every non-zero entry in that row is $\pm 1$, then we would simply need to multiply by $\pm 1$. Thus, we shall instead prove the stronger hypothesis that at every step of Gaussian elimination, the intermediate matrix is flat, and hence all of its non-zero entries are $\pm 1$. This is the content of Lemma \ref{lemma_flat}.
\end{proof}

Before proving Lemma \ref{lemma_flat}, we include an example to illustrate the method. Here we omit row swapping for clarity, and we obtain a permutation matrix, which has determinant $\pm 1$. At each step, the leading nonzero term in the bolded row is used to clear the corresponding column.
\begin{gather}
A_0:
 \begin{bmatrix}
0 & 1 & 0 & 0 & 0 \\
0 & 0 & 0 & 0 & 1 \\
\mathbf{1} & \mathbf{-1} & \mathbf{0} & \mathbf{0} & \mathbf{0} \\
0 & -1 & 1 & -1 & 1 \\
0 & 0 & 0 & -1 & 1
\end{bmatrix} \to \
A_1: \begin{bmatrix}
\mathbf{0} & \mathbf{1} & \mathbf{0} & \mathbf{0} & \mathbf{0} \\
0 & 0 & 0 & 0 & 1 \\
1 & -1 & 0 & 0 & 0 \\
0 & -1 & 1 & -1 & 1 \\
0 & 0 & 0 & -1 & 1
\end{bmatrix} \to \
A_2: \begin{bmatrix}
0 & 1 & 0 & 0 & 0 \\
0 & 0 & 0 & 0 & 1 \\
1 & 0 & 0 & 0 & 0 \\
\mathbf{0} & \mathbf{0} & \mathbf{1} & \mathbf{-1} & \mathbf{1} \\
0 & 0 & 0 & -1 & 1
\end{bmatrix} \\
\to \
A_3: \begin{bmatrix}
0 & 1 & 0 & 0 & 0 \\
0 & 0 & 0 & 0 & 1 \\
1 & 0 & 0 & 0 & 0 \\
0 & 0 & 1 & -1 & 1 \\
\mathbf{0} & \mathbf{0} & \mathbf{0} & \mathbf{-1} & \mathbf{1}
\end{bmatrix} \to \
A_4: \begin{bmatrix}
0 & 1 & 0 & 0 & 0 \\
\mathbf{0} & \mathbf{0} & \mathbf{0} & \mathbf{0} & \mathbf{1} \\
1 & 0 & 0 & 0 & 0 \\
0 & 0 & 1 & 0 & 0 \\
0 & 0 & 0 & 1 & -1
\end{bmatrix} \to \
A_5: \begin{bmatrix}
0 & 1 & 0 & 0 & 0 \\
0 & 0 & 0 & 0 & 1 \\
1 & 0 & 0 & 0 & 0 \\
0 & 0 & 1 & 0 & 0 \\
0 & 0 & 0 & 1 & 0
\end{bmatrix}.
\end{gather}

\begin{lemma}\label{lemma_flat}
In carrying out Gaussian elimination on the matrix $A$ as in Theorem \ref{theorem:cube}, all intermediate matrices are flat.
\end{lemma}
\begin{proof}
We proceed by induction. Let $A_k$ denote the matrix resulting from the $k$\textsuperscript{th} step of Gaussian elimination (i.e. the matrix obtained after ``clearing'' the first $k$ columns). We shall show that for each $k$, every row of the matrix $A_k$ is of exactly one of six types depending on the form of the first $k$ entries of that row and the last $m - k$ entries of that row (in the sequel, we will refer to this as saying that every row is one of the six types with respect to $k$).

We now describe these six types. Let $\alpha_n$ denote any sequence of length $n$ consisting of alternating plus ones and minus ones (e.g. $\alpha_3 = [-1, 1, -1]$ or $\alpha_1 = [1]$). Let $\beta_n$ denote the sequence of length $n$ consisting of all zeros. Let $\gamma_n$ denote any binary sequence of length $n$ containing exactly one one (e.g. $\gamma_4 = [0, 0, 1, 0]$). Let $\oplus$ refer to the operation of vector concatenation (e.g. $[1, 2, 3] \oplus [4, 5] = [1, 2, 3, 4, 5]$). The six types (with respect to $k$) are listed in Table \ref{table:firsttable}f.

\begin{table}[h]
\begin{center}
\begin{tabular}{l l l l}
Type 	&	First $k$	&	Last $m - k$	&	Example ($k = 3$, $m = 7$) \\
\hline
1	&	$\beta_k$	&	$\beta_{\ell \geq 1} \ \oplus \ \alpha_{j \geq 1} \ \oplus \beta_{m - k - \ell - j \geq 1}$	&	$[0, 0, 0 \ \big| \ 0, 1, -1, 0]$ \\
2	&	$\beta_k$	&	$\alpha_{\ell \geq 1} \ \oplus \ \beta_{n - k - \ell \geq 0}$	&	$[0, 0, 0 \ \big| \ 1, -1, 1, 0]$ \\
3	&	$\beta_k$	&	$\beta_{\ell \geq 1} \ \oplus \ \alpha_{m-k-\ell \geq 0}$	&	$[0, 0, 0 \ \big| \ 0, 0, 1, -1]$ \\
4 	&	$\gamma_k$	&	$\beta_{\ell \geq 1} \ \oplus \alpha_{j \geq 0} \ \oplus \beta_{m-k-\ell-j \geq 1}$	&	$[0, 1, 0 \ \big| \ 0, 0, 0, 0]$ \\
5	&	$\gamma_k$	&	$\alpha_{\ell \geq 1} \ \oplus \ \beta_{n - k - \ell \geq 0}$	&	$[0, 1, 0 \ \big| \ 1, -1, 1, 0]$ \\
6	&	$\gamma_k$	&	$\beta_{\ell \geq 1} \ \oplus \ \alpha_{m-k-\ell \geq 0}$	&	$[0, 1, 0 \ \big| \ 0, 0, 1, -1]$ \\
& & &
\end{tabular}
\caption{\label{table:firsttable} The six types with respect to $k$}
\end{center}
\end{table}

We now go through the inductive argument. For the base case, notice that when $k = 0$, the cube vectors are type 1, the left ballot vectors are type 2, and the right  ballot vectors are type 3. Thus the claim is proven in the base case.

Now for the inductive step, we shall show that if all rows of $A_k$ are of one of the above types with respect to $k$, then all rows of $A_{k+1}$ are of one of the above types with respect to $k+1$. As described in the proof of Theorem \ref{theorem:cube}, at step $k$ we must first find some row whose first $k$ entries are zero, and whose $k+1$ entry is $\pm 1$. We see then that we must select some row of type 2, call it $T$. We then subtract $T$ from all other rows whose $k+1$ entry is non-zero. Thus the only types we must worry about are types 2 and 5. Notice that when we subtract $T$ from a row of type 2, we get a row either or type 1, type 2, or type 3 with respect to $k+1$. When we subtract $T$ from a row of type 5, we get a row either of type 4, 5, or 6 with respect to $k+1$. All other rows remain the same. Thus when we catalog the new rows with respect to $k+1$, we get that those of type 1 become either type 1 or type 2. As mentioned before, those of type 2 become those of type 1, 2, or 3, except for row $T$ which becomes of type 4 or 5. Type 3 becomes type 2 or 3. Type 4 remains type 4 or becomes type 5. As mentioned before, type 5 becomes type 4, 5, or 6. Lastly, type 6 becomes type 5 or type 6. Thus, by induction, we have proven the desired statement, implying in particular that the matrix is flat at every step.
\end{proof}

\section{Vertices of the cube in the ballot region} \label{section:polytope_vertices}

In this section, we demonstrate that bidirectional ballot sequences of length $2n-1$ correspond in a natural way to $\Rn$, and we rederive the growth rate given in \cite{Zh1} and \cite{BP}.


\begin{definition}
	A \textbf{slope vector} is a vector $\lambda = [\lambda_1,\dots,\lambda_m] \in \R^m$ with $m\in \N$. To a slope vector $\lambda$, we associate the unique continuous piecewise linear function $f_\lambda: [0,m] \to \R$ such that $f(0) = 0$ and $f_\lambda'(x) = \lambda_i$ for $x\in (i-1,i)$ for each $1\le i\le m$.
\end{definition}

Given any binary sequence $b = b_1\cdots b_m$, we associate to this sequence the graph of the function $f_\lambda$ where $\lambda = (\lambda_1,\dots,\lambda_m)$ with $\lambda_i \coloneqq (-1)^{b_i-1}$.

\begin{example}\label{ex:discrete_graph}
The bidirectional ballot sequence $11011001111$ corresponds to the path
\begin{center}
\begin{tikzpicture}


\draw[->] (0,0) -- (11,0) node[anchor=north] {};
\draw[->, dashed] (0,15/3) -- (11,15/3) node[] {};

\draw[->] (0,0) -- (0,5.5) node[anchor=east] {};
		
\draw	(0,0) node[circle, fill, inner sep=2pt] {}
		(1,1) node[circle, fill, inner sep=2pt] {}
		(2,2) node[circle, fill, inner sep=2pt] {}
		(3,1) node[circle, fill, inner sep=2pt] {}
		(4,2) node[circle, fill, inner sep=2pt] {}
		(5,3) node[circle, fill, inner sep=2pt] {}
		(6,2) node[circle, fill, inner sep=2pt] {}
		(7,1) node[circle, fill, inner sep=2pt] {}
		(8,2) node[circle, fill, inner sep=2pt] {}
		(9,3) node[circle, fill, inner sep=2pt] {}
		(10,4) node[circle, fill, inner sep=2pt] {}
		(11,5) node[circle, fill, inner sep=2pt] {};


\draw[thick] (0,0) -- (1,1) -- (2,2) -- (3,1) -- (4,2) -- (5,3) -- (6,2) -- (7,1) -- (8,2) -- (9,3) -- (10,4) -- (11,5);



\end{tikzpicture}
\end{center}
\end{example}

This is a bijection from binary sequences of length $m$ to graphs of functions $f_\lambda$ with $\lambda \in \{\pm 1\}^m$. Recall from Section \ref{sec:intr} that the graphs which correspond to bidirectional ballot sequences are those of functions $f_\lambda$ where $f_\lambda(0) < f_\lambda(t) < f_\lambda(m)$ for all $0 < t < m$.

Now we will draw a correspondence between $\Rn$ and $B_{2n+3}$ through these graphs, as well as a correspondence between a certain subset of $\Rn$ and $B_{2n-1}$, by describing a way to interpret vectors $v\in C_{2n-1} = [0,1]^{2n-1}$ as paths as in the discrete case in such a way that the vertices of the ballot polytope are realized as exactly the graphs above. Given a vector $v = [v_1,\dots,v_{2n-1}] \in  C_{2n-1}$, define the slope vector $\lambda_v = [\lambda_1,\dots,\lambda_{2n-1}]$ by $\lambda_i \coloneqq (-1)^{i-1} (2v_i-1)$, and associate to $v$ the graph of the function $f_{\lambda_v}$.

\begin{example}\label{ex:graph}
The gap-parametrization vector $v= \left[\frac34, \frac13, \frac12, \frac23, 1\right] \in [0,1]^5$ gives the slope vector $\l_v = \left[\frac12, -\frac13, 0, \frac13, 1 \right]$, which gives the following graph of the function $f_{\l_v}$, where the values next to the points indicate the distance above the $x$-axis:

\begin{center}
\begin{tikzpicture}


\draw[->] (0,0) -- (11,0) node[anchor=north] {};
\draw[->, dashed] (0,6) -- (11,6) node[] {};
\draw	(0,-0.5) node[] {$0$}
		(2,-0.5) node[] {$1$}
		(4,-0.5) node[] {$2$}
		(6,-0.5) node[] {$3$}
		(8,-0.5) node[] {$4$}
		(10,-0.5) node[] {$5$};

\draw[->] (0,0) -- (0,6.5) node[anchor=east] {};
		
\draw	(0,0) node[circle, fill, inner sep=2pt] {}
		(2,2) node[circle, fill, inner sep=2pt](a) {}
		(4,4/6) node[circle, fill, inner sep=2pt](b) {}
		(6,4/6) node[circle, fill, inner sep=2pt](c) {}
		(8,2) node[circle, fill, inner sep=2pt](d) {}
		(10,6) node[circle, fill, inner sep=2pt](e) {};
		\node[left=0.1cm of a] {\tiny $\frac12$};
		\node[above=0.05cm of b] {\tiny $\frac16$};
		\node[above=0.05cm of c] {\tiny $\frac16$};
		\node[left=0.1cm of d] {\tiny $\frac12$};
		\node[below=0.05cm of e] {\tiny $\frac32$};


\draw[thick] (0,0) -- (2,2) -- (4,4/6) -- (6,4/6) -- (8,2) -- (10,6);



\end{tikzpicture}
\end{center}
\end{example}

Although the function $f_{\l_v}$ in Example \ref{ex:graph} has the property that it achieves global minimum and maximum values at it left and right endpoints (respectively), we will see that this is not always the case (see Example \ref{ex:process}). We determine this behavior more precisely now.

If $v = [v_1,\dots,v_{2n-1}] \in  C_{2n-1}$, then for $0\le k\le 2n-1$ we have
\begin{equation}\label{eq:gpfunc1}
    f_{\lambda_v}(k)\ =\ \sum_{j=1}^k (-1)^{j-1}(2v_j-1) = \begin{cases}
        2\sum_{j=1}^k (-1)^{j-1} v_j & \mbox{$k$ is even} \\
        -1 + 2\sum_{j=1}^k (-1)^{j-1} v_j & \mbox{$k$ is odd,}
    \end{cases}
\end{equation}
and similarly
\begin{equation}\label{eq:gpfunc2}
    f_{\lambda_v}(2n-1-k)\ =\ \begin{cases}
        f_{\lambda_v}(2n-1) - 2\sum_{j=1}^k (-1)^{j-1} v_{2n-j} & \mbox{$k$ is even} \\
        f_{\lambda_v}(2n-1) + 1 - 2\sum_{j=1}^k (-1)^{j-1} v_{2n-j} & \mbox{$k$ is odd.}
    \end{cases}
\end{equation}
One can see now that, even if $v\in \mP_n$, it is possible for the graph to fail the property stated above, i.e., to achieve a global maximum or minimum at a point in the interior of its interval of definition (again, see Example \ref{ex:process} for an explicit example). However, one can also see that if $v\in\mP_n$, it cannot fail this property to a great extent; namely, the values at the left and right endpoints will be within a distance of 1 from the maximum and minimum values, since the large sums in the RHS of \eqref{eq:gpfunc1} and \eqref{eq:gpfunc2} will be non-negative. Nonetheless, we would like the graphs of the functions $f_{\l_v}$ with $v\in\Rn$ to match the graphs of bidirectional ballot sequences in $B_{2n+3}$, and for that reason we give a way to modify a vector $v\in\Rn$ before associating it to a graph. Namely, we will add a sort of buffer to each side of the vector, so that the left and right endpoints get a leg up.

\begin{definition}
	If $v=[v_1,\dots,v_{2n-1}]\in  C_{2n-1}$, we define
\begin{equation*}
    \alpha(v)\ \coloneqq\ [1,0,v_1,v_2,\dots,v_{2n-2},v_{2n-1},0,1].
\end{equation*}
\end{definition}

We now present two correspondences, the first stated more naturally, and the second proven more naturally, which are nonetheless very closely related. The first correspondence is as follows.

\begin{theorem}\label{thm:ballot_vertices} The set
$\Rn$ is in bijection with $B_{2n+3}$, induced by the map
\begin{equation}\label{eq:bijection}
v \mapsto f_{\lambda_{\alpha(v)}}.
\end{equation}
\end{theorem}

Before we prove Theorem \ref{thm:ballot_vertices}, we give an example of the process that induces the bijection.

\begin{example}\label{ex:process}
Consider the gap-parametrization vector $v = [0,0,1,0,0] \in [0,1]^5$, an element of $Q_3$. We shall obtain a bidirectional ballot sequence from $v$.
\emph{
We see that $v$ gives the slope vector $\l_v = [-1,1,1,1,-1]$. The graph of $f_{\l_v}$ is the following, where the values next to the points indicate the distance above the $x$-axis:
}

\begin{center}
\begin{tikzpicture}


\draw[->] (0,2) -- (11,2) node[anchor=north] {};
\draw[->, dashed] (0,4) -- (11,4) node[] {};
\draw	(2, 1.5) node[] {$1$}
		(4, 1.5) node[] {$2$}
		(6, 1.5) node[] {$3$}
		(8, 1.5) node[] {$4$}
		(10, 1.5) node[] {$5$};

\draw[<->] (0,-0.5) -- (0,7) node[anchor=east] {};
\draw	(-0.5,2) node[] {$0$};
		
\draw	(0,2) node[circle, fill, inner sep=2pt] {}
		(2,0) node[circle, fill, inner sep=2pt](a) {}
		(4,2) node[circle, fill, inner sep=2pt](b) {}
		(6,4) node[circle, fill, inner sep=2pt](c) {}
		(8,6) node[circle, fill, inner sep=2pt](d) {}
		(10,4) node[circle, fill, inner sep=2pt](e) {};
		\node[left=0.1cm of a] {\tiny $-1$};
		\node[above=0.05cm of b] {\tiny $0$};
		\node[above=0.05cm of c] {\tiny $1$};
		\node[left=0.1cm of d] {\tiny $2$};
		\node[below=0.05cm of e] {\tiny $1$};


\draw[thick] (0,2) -- (2,0) -- (4,2) -- (6,4) -- (8,6) -- (10,4);



\end{tikzpicture}
\end{center}
\emph{
This is \emph{not} the graph of a bidirectional ballot sequence. Namely, the graph passes below the $x$-axis and above the line $y = f_{\l_v}(5)$. Let's now consider $\alpha(v) = [1,0,0,0,1,0,0,0,1] \in [0,1]^9$, which gives slope vector $\l_{\alpha(v)} = [1,1,-1,1,1,1,-1,1,1]$ and leads to the following graph of $f_{\l_{\alpha(v)}}$.
}
\begin{center}
\begin{tikzpicture}


\draw[->] (0,0) -- (10,0) node[anchor=north] {};
\draw[->, dashed] (0,5) -- (10,5) node[] {};
\draw	(0, -0.5) node[] {$0$}
		(1, -0.5) node[] {$1$}
		(2, -0.5) node[] {$2$}
		(3, -0.5) node[] {$3$}
		(4, -0.5) node[] {$4$}
		(5, -0.5) node[] {$5$}
		(6, -0.5) node[] {$6$}
		(7, -0.5) node[] {$7$}
		(8, -0.5) node[] {$8$}
		(9, -0.5) node[] {$9$};

\draw[->] (0,0) -- (0,6) node[anchor=east] {};
\draw[dashed] (2,0) -- (2,5) node[] {};
\draw[dashed] (7,0) -- (7,5) node[] {};
		
\draw	(0,0) node[circle, fill, inner sep=2pt] {}
		(1,1) node[circle, fill, inner sep=2pt](a) {}
		(2,2) node[circle, fill, inner sep=2pt](b) {}
		(3,1) node[circle, fill, inner sep=2pt](c) {}
		(4,2) node[circle, fill, inner sep=2pt](d) {}
		(5,3) node[circle, fill, inner sep=2pt](e) {}
		(6,4) node[circle, fill, inner sep=2pt](f) {}
		(7,3) node[circle, fill, inner sep=2pt](g) {}
		(8,4) node[circle, fill, inner sep=2pt](h) {}
		(9,5) node[circle, fill, inner sep=2pt](i) {};
		\node[above=0.05cm of a] {\tiny $1$};
		\node[right=0.05cm of b] {\tiny $2$};
		\node[below=0.05cm of c] {\tiny $1$};
		\node[above=0.05cm of d] {\tiny $2$};
		\node[above=0.05cm of e] {\tiny $3$};
		\node[above=0.05cm of f] {\tiny $4$};
		\node[left=0.05cm of g] {\tiny $3$};
		\node[above=0.05cm of h] {\tiny $4$};
		\node[above=0.05cm of i] {\tiny $5$};


\draw[thick] (0,0) -- (1,1) -- (2,2) -- (3,1) -- (4,2) -- (5,3) -- (6,4) -- (7,3) -- (8,4) -- (9,5);



\end{tikzpicture}
\end{center}
\emph{
The portion of the graph between the vertical dotted lines is simply the graph of $f_{\l_v}$ translated in the plane by the vector $[2,2]$. This graph \emph{does} correspond to a bidirectional ballot sequence, namely $110111011$. We now prove that this process gives a bijection as in the statement of the theorem.
}
\end{example}

\begin{proof}[Proof of Theorem \ref{thm:ballot_vertices}]
By the correspondence between bidirectional ballot sequences and graphs of certain functions given in Example \ref{ex:discrete_graph}, it suffices to show that the map of \eqref{eq:bijection} puts $\Rn$ in bijection with
\begin{equation}\label{eq:defF}
    F\ =\ \{f_\mu: \mu \in \{\pm 1\}^{2n+3},\; f_\mu(0)\ <\ f_\mu(t)\ <\ f_\mu(2n+3)\ \text{ for all }\ t\in(0,2n+3)\}.
\end{equation}
If $v\in C_{2n-1}$ is any gap-parametrization vector, then, in light of \eqref{eq:gpfunc1}, \eqref{eq:gpfunc2}, and the fact that $f_{\lambda_v}$ achieves maxima and minima only at integer values, we have that $f_{\lambda_v}(0)-1 \le f_{\lambda_v}(t) \le f_{\lambda_v}(2n-1)+1$ for $t\in [0,2n-1]$ if and only if $v$ is a bidirectional gerrymander. Furthermore, if $v$ is a vertex of the cube $C_{2n-1}$, then $\alpha(v)$ is a vertex of $C_{2n+3} = [0,1]^{2n+3}$ so that $f_{\l_{\alpha(v)}}$ takes integers to integers. Since for any $v\in C_{2n-1}$ we have $f_{\l_{\alpha(v)}}(k+2) = f_{\l_v}(k) + 2$ for $0\le k \le 2n-1$, $f_{\l_{\alpha(v)}}(i) = i$ for $i=0,1,2$, and $f_{\l_{\alpha(v)}}(2n+1+i) = f_{\l_{\alpha(v)}}(2n+1)+i$ for $i=1,2$. Thus if $v$ is a vertex of $ C_{2n-1}$ then $f_{\l_{\alpha(v)}}(0) < f_{\l_{\alpha(v)}}(t) < f_{\l_{\alpha(v)}}(2n+3)$ for all $t\in(0,2n+3)$ if and only if $v\in \Rn$. It follows then that, since $\lambda_{\alpha(v)} \in \{\pm 1\}^{2n+3}$ when $v\in \Rn$, we indeed have that $f_{\l_{\alpha(v)}} \in F$, and so the map in \eqref{eq:bijection} does indeed take $\Rn$ to graphs of bidirectional ballot sequences in $B_{2n+3}$.

Injectivity of the map is clear. To show that the map is surjective, we provide an inverse. For a bidirectional ballot sequence $b=b_1\cdots b_{2n+3}$ of length $2n+3$, we define the vector $w = [w_1,\dots,w_{2n-1}]$, where
\begin{equation}
    w_j\ \coloneqq\ \begin{cases}
        1 & \mbox{if $j \equiv b_{j+2} \pmod 2$} \\
        0 & \mbox{if $j \not\equiv b_{j+2} \pmod 2$.}
    \end{cases}
\end{equation}
It is easily verified that the graph of $f_{\l_{\alpha(w)}}$ is the one associated to $b$. Moreover, the two statements directly following \eqref{eq:defF} imply that, since $w\in\{\pm 1\}^{2n-1}$ and the graph of $f_{\l_{\alpha(w)}}$ is that of a bidirectional ballot sequence, we must have that $w\in \Rn$. It is clear that this map is both a right- and left-inverse of the map given by \eqref{eq:bijection}.
\end{proof}

We now give the second correspondence. Let $\mathcal{I}_n$ denote the interior of $\mathcal{B}_n$ in $\R^{2n-1}$. Let $T_n = \mathcal{I}_n \cap \Rn$, i.e. those vertices of $\mP_n$ in the interior of $\mathcal{B}_n$.
\begin{corollary}\label{cor:ballot_vertices}
    We have $T_n$ is in bijection with $B_{2n-1}$, induced by the map
    \begin{equation}
        v\ \mapsto\ f_{\l_v}.
    \end{equation}
\end{corollary}
\begin{proof}
    The proof here is essentially the same as that of Theorem \ref{thm:ballot_vertices}. The point here is that, when $v\in T_n$, we already have $f_{\l_v}(0) < f_{\l_v}(t) < f_{\l_v}(2n-1)$, following similar reasoning as in the statements directly following \eqref{eq:defF}.
\end{proof}

Lastly, we use these correspondences along with our previous analysis of $\mP_n$ and its translates to obtain the growth rate in \cite{Zh1}.

\begin{corollary}\label{cor:upper}
For $\ell$ odd,
\begin{gather}
    B_\ell\ \geq\ \frac{2^\ell}{16 (\ell-4)}.
\end{gather}
\end{corollary}

\begin{proof}
The inequality is trivial if $\ell \in \{1,3\}$, so assume $\ell \geq 5$. Let $m = \ell - 4$; this is $2n-1$ for some $n \in \N$. By Theorem \ref{thm:ballot_vertices}, we know that the vertices of $\mathcal{P}_n$ are in bijection with $B_{m+4}$. From Corollary \ref{cor:volume}, we know that every vertex of $C_{2n-1}$ is contained in $\mathcal{P}_\sigma$ for some $\sigma \in Z_m$. Since there are $m$ such copies of $\mathcal{P}$, we have
\begin{gather}
    m B_{m+4}\ \geq\ 2^m.
\end{gather}
By rearrangement we get
\begin{gather}
    B_\ell\ \geq\ \frac{2^\ell}{16 (\ell-4)}.
\end{gather}
\end{proof}

\begin{corollary}\label{cor:lower}
For $\ell$ odd,
\begin{gather}
    B_\ell\ \leq\ \frac{2^\ell}{\ell}.
\end{gather}
\end{corollary}

\begin{proof}
Suppose $\ell = 2n-1$. From Corollary \ref{cor:ballot_vertices}, we know that the vertices of $\mathcal{P}_n$ which are in the interior of $\mathcal{B}_n$, namely $T_n$, are in bijection with $B_m$. Since the interiors of $\mathcal{B}_{\sigma_1}$ and $\mathcal{B}_{\sigma_2}$ are disjoint if $\sigma_1 \neq \sigma_2$, we have that $\sigma_1 (T_n) \cap \sigma_2 (T_n) = \emptyset$ for $\sigma_1 \neq \sigma_2$. Therefore, summing over all the vertices in $\sigma(T)$ for each $\sigma \in Z_\ell$, we at most get every vertex of the cube once. That is,
\begin{gather}
\ell B_\ell\ \leq\ 2^\ell.
\end{gather}
Rearranging yields
\begin{gather}
B_\ell\ \leq\ \frac{2^\ell}{\ell}.
\end{gather}
\end{proof}

\begin{corollary}
For all $\ell$, the growth rate of $B_\ell$ is $\Theta(2^\ell/\ell)$.
\end{corollary}
\begin{proof}
By Corollaries \ref{cor:upper} and \ref{cor:lower}, we know that for $\ell$ odd, the growth rate is $\Theta(2^\ell/\ell)$. The only additional insight needed is that for all $\ell$, $B_{\ell+1} \geq B_\ell$. To see this, note that given a BBS of length $\ell$, by appending a 1 to the end of it, we obtain a BBS of length $\ell+1$. Thus up to fixed constants, the inequalities in Corollaries \ref{cor:upper} and \ref{cor:lower} are correct for even $\ell$ as well. Thus, for all $\ell$, $B_\ell$ grows like $\Theta(2^\ell/\ell)$.
\end{proof}

\section{Conclusion}

Our methods reveal a rich combinatorial structure underlying bidirectional ballot sequences. In previous papers on BBS's (\cite{Zh1}, \cite{BP}, \cite{HHPW}), analytic techniques were used to obtain asymptotics, but our techniques reveal a geometric interpretation for the $\Theta (2^n/n)$ growth rate. Interestingly, in the final section of \cite{Zh1}, Zhao states without detailed proof that $n B_n/2^n$ goes to $1/4$, but claims his proof is ``calculation-heavy''. He then posits that ``[t]here should be some natural, combinatorial explanation, perhaps along the lines of grouping all possible walks into orbits of size mostly $n$ under some symmetry, so that almost every orbit contains exactly one walk with the desired property.'' Zhao's statement is strikingly similar to the ideas presented in our paper. Though we have made some effort, we have not been able to derive that $n B_n/2^n \to 1/4$ using the techniques of our paper, but we feel that there is hope for such a proof.

The second, more general takeaway from this paper is the potential for the ideas originally presented in \cite{MP}. The ideas in this paper in fact evolved from the ideas in \cite{MP}. In passing to the continuous setting, several additive number theory and combinatorial problems reveal a rich structure which was not otherwise visible. We believe that there is even greater potential still in such ideas and techniques.

\bigbreak


\end{document}